\documentclass[a4paper]{amsart}
\usepackage{amsfonts}
\usepackage{amssymb,latexsym}

\theoremstyle{plain}
\newtheorem{theorem}{Theorem}

\newtheorem{lemma}[theorem]{Lemma}

\theoremstyle{definition}
\newtheorem{definition}[theorem]{Definition}
\theoremstyle{remark}

\newtheorem*{remark}{Remark}
\newtheorem*{example}{Example}
\numberwithin{equation}{section}
\numberwithin{theorem}{section}

\begin{document}
\title[Minimal clones with weakly abelian representations]{Minimal clones
with weakly abelian representations}
\author[T. Waldhauser]{Tam\'{a}s Waldhauser}
\address{Bolyai Institute\\
University of Szeged\\
Aradi v\'{e}rtan\'{u}k tere 1, H6720, Szeged, Hungary}
\email{twaldha@math.u-szeged.hu}
\thanks{{Research supported by the Hungarian National Foundation for
Scientific Research grant no. T 026243 and the Research Group on Artificial
Intelligence, HAS-SZTE}}

\begin{abstract}
We show that a minimal clone has a nontrivial weakly abelian representation
iff it has a nontrivial abelian representation, and that in this case all
representations are weakly abelian.
\end{abstract}

\dedicatory{Dedicated to B\'{e}la Cs\'{a}k\'{a}ny on his seventieth birthday}
\maketitle

\section{Introduction}

\bigskip

A \textit{concrete clone} is a composition-closed collection of operations
on some set containing all the projections. An \textit{abstract clone} is a
heterogeneous algebra equipped with operations which mimic the composition
operations of concrete clones. (For the formal definition see \cite{Taylor}
or \cite{LLPPP}.)

A \textit{representation} of an abstract clone is a homomorphism into the
concrete clone of operations on a given set. Usually one obtains a
representation by picking a set of generators of the clone and assigning to
each of them an operation of the same arity on a set in such a way that this
assignment extends to a clone homomorphism. Thus each representation gives
an algebra, and these algebras form a variety. (If we choose another set of
generators, then we get another variety which is term-equivalent to the
previous one.) Conversely, every variety arises in this way from the clone
of term functions of the countably generated free algebra in the variety.

A clone is \textit{minimal} if it has exactly two subclones: the clone
itself and the clone which consists of projections only. The latter is
called a \textit{trivial clone}, and in this paper we will call an algebra
trivial if the clone of its term functions is trivial (even if the algebra
has more than one element!). Specially, a groupoid is trivial iff it is a
left or right zero semigroup. A nontrivial representation of a minimal clone
is also minimal, so if a variety has a minimal clone, then any nontrivial
algebra in the variety has a minimal clone.

Let us now recall the definition of four variants of abelianness (cf.\cite%
{Keith-Emil}). For an algebra $\mathbb{A}$ let $\mathcal{M}({\mathbb{A}})$
denote the set of $2\times 2$ matrices of the form $\Bigl(%
\begin{smallmatrix}
{t(\mathbf{a},\mathbf{c})} & {t(\mathbf{a},\mathbf{d})} \\ 
{t(\mathbf{b},\mathbf{c})} & {t(\mathbf{b},\mathbf{d})}%
\end{smallmatrix}%
\Bigr)$ where $t$ is a polynomial of $\mathbb{A}$ of arity $n+m$ and $%
\mathbf{a},\mathbf{b}\in A^{n},\ \mathbf{c},\mathbf{d}\in A^{m}$.

\begin{definition}
\label{DEF 1.1}We say that an algebra $\mathbb{A}$ is

\begin{enumerate}
\item \label{weakly abelian}\emph{weakly abelian} if $\bigl(%
\begin{smallmatrix}
u & u \\ 
u & v%
\end{smallmatrix}%
\bigr)\in \mathcal{M}(\mathbb{A})$ implies $u=v$;

\item \emph{abelian} if $\bigl(%
\begin{smallmatrix}
u & u \\ 
v & w%
\end{smallmatrix}%
\bigr)\in \mathcal{M}(\mathbb{A})$ implies $v=w$;

\item \emph{rectangular} if $\bigl(%
\begin{smallmatrix}
u & v \\ 
w & u%
\end{smallmatrix}%
\bigr)\in \mathcal{M}(\mathbb{A})$ implies $u=v=w$;

\item \emph{strongly abelian} if it is both abelian and rectangular.
\end{enumerate}
\end{definition}

All of these properties are inherited by subalgebras and direct products,
but not by homomorphic images. If $\mathbb{A}$ is a groupoid, and we apply (%
\ref{weakly abelian}) to $t(x,y)=xy$ then we get that whenever in the
multiplication table of $\mathbb{A}$ we see a configuration like this: 
\begin{equation*}
\begin{tabular}{l|lllll}
& $\cdots $ & $c$ & $\cdots $ & $d$ & $\cdots $ \\ \hline
$\vdots $ &  & $\vdots $ &  & $\vdots $ &  \\ 
$a$ & $\cdots $ & $u$ & $\cdots $ & $u$ & $\cdots $ \\ 
$\vdots $ &  & $\vdots $ &  & $\vdots $ &  \\ 
$b$ & $\cdots $ & $u$ & $\cdots $ & $v$ & $\cdots $ \\ 
$\vdots $ &  & $\vdots $ &  & $\vdots $ & 
\end{tabular}%
\end{equation*}%
then we must have $u=v$. Of course, this is just a necessary condition for $%
\mathbb{A}$ to be weakly abelian.

Minimal clones with abelian representations have been described by K.
Kearnes in \cite{KK}. Here we examine the analogous question for the other
three concepts. We will show that if a minimal clone has a nontrivial weakly
abelian representation, then it also has a nontrivial abelian
representation, and all representations are weakly abelian. From this result
we will easily deduce that if a minimal clone has a nontrivial rectangular
representation, then it also has a nontrivial strongly abelian
representation; moreover, all representations are strongly abelian.

\bigskip

\section{Preliminary results}

\bigskip

Minimal clones are generated by any of their nontrivial elements and it is
convenient to choose one of minimum arity. Such a generator must be one of
five types according to the following theorem of Rosenberg \cite{R 5typ}
(see also \cite{SzA clUA}).

\begin{theorem}
\label{THM 2.1}%
{\upshape{(\cite{R 5typ}).}}
Let $f$ be a nontrivial operation of minimum arity in a minimal clone. Then $%
f$ satisfies one of the following conditions:%
\renewcommand{\theenumi}{(\Roman{enumi})}
\renewcommand{\labelenumi}{\theenumi}%
\begin{enumerate}
\item \label{unary}$f$ is unary, and $f^{2}(x)=f(x)$, or $f^{p}(x)=x$ for
some prime $p$;

\item \label{binary}$f$ is a binary idempotent operation, i.e. $f(x,x)=x$;

\item \label{majority}$f$ is a ternary majority operation, i.e. $%
f(x,x,y)=f(x,y,x)=f(y,x,x)=\nolinebreak x$;

\item \label{minority}$f(x,y,z)=x+y+z$ for an elementary abelian 2-group
with addition $+$;

\item \label{semiprojection}$f$ is a semiprojection, i.e. there exists an $%
i\ (1\leq i\leq n)$ such that $\allowbreak f(x_{1},x_{2},\ldots
,x_{n})=x_{i} $ whenever the arguments are not pairwise distinct.
\end{enumerate}
\end{theorem}

A minimal clone cannot contain operations of two different types, therefore
we can speak about five types of minimal clones. Any representation of a
clone of type \ref{unary} is strongly abelian; any nontrivial representation
of a clone of type \ref{minority} is abelian, but not rectangular (hence not
strongly abelian). A minimal clone of type \ref{majority} or \ref%
{semiprojection} cannot have a nontrivial weakly abelian representation.
This is shown in Theorem 3.1 in \cite{KK}. (This theorem is about abelian
representations, but the proof actually shows that there is no weakly
abelian representation either.) Thus we have to consider clones of type \ref%
{binary} only.

To recall the results of \cite{KK}, we have to define several clones. By the
clone of an \textit{affine space} we mean the clone of all idempotent term
functions of a vector space over some field. This clone is minimal iff the
field is a $p$-element field for some prime number $p$. If $p>2$, then this
clone is of type \ref{binary}: any nontrivial operation of the form $\lambda
x+(1-\lambda )y$ generates the clone. If $p=2$ then the clone is of type \ref%
{minority}: the minority operation $x+y+z$ is a generator of minimum arity.

For any prime $p$, let us define the variety of \textit{$p$-cyclic groupoids}
by the identities $xx=x,x(yz)=xy,(xy)z=(xz)y,(\cdots ((xy)y)\cdots
)y=xy^{p}=x$. These groupoids have been introduced by P\l onka \cite{Plonka
k-cyc}; he also proved that they have minimal clones \cite{Plonka idred}. 
\textit{Rectangular bands} are idempotent semigroups satisfying $xyz=xz$,
and they have minimal clones, too.

Now we can describe all minimal clones with a nontrivial abelian
representation (Theorem 3.11 in \cite{KK}).

\begin{theorem}
\label{THM 2.2}%
{\upshape{(\cite{KK}).}}
The minimal clones which have a nontrivial abelian representation are the
following:%
\renewcommand{\theenumi}{(\roman{enumi})}
\renewcommand{\labelenumi}{\theenumi}%

\begin{enumerate}
\item the unary clone generated by an operation $f$ satisfying $f(x)=f(y)$,
but not satisfying $f(x)=x$;

\item the unary clone generated by an operation $f$ satisfying $%
f^{2}(x)=f(x) $, but not satisfying $f(x)=f(y)$ or $f(x)=x$;

\item the unary clone generated by an operation $f$ satisfying $f^{p}(x)=x$
for some prime $p$, but not satisfying $f(x)=x$;

\item the clone of any nontrivial rectangular band;

\item the clone of an affine space over a prime field;

\item the clone of any nontrivial $p$-cyclic groupoid (or its dual) for some
prime\nolinebreak\ $p$.
\end{enumerate}
\end{theorem}

The following interesting property of abelian representations has also been
proved in \cite{KK} with the help of absorption identities (see also \cite%
{LLPPP}).

\begin{theorem}
\label{THM 2.3}%
{\upshape{(\cite{KK}).}}
If a minimal clone has a nontrivial abelian representation, then this
representation is faithful.
\end{theorem}

As a special case of this theorem we have that if a variety $\mathcal{V}$
has a minimal clone and it contains a nontrivial rectangular band or affine
space, then $\mathcal{V}$ must be the variety of rectangular bands or a
variety of affine spaces. From the proof it is clear that the same is true
for $p$-cyclic groupoids too, although not all of them are abelian, as we
will see in the last section.

\bigskip

\section{Weak abelianness and distributivity}

\bigskip

In the theory of groupoids and quasigroups a different notion of `weak
abelianness' is defined by the identities 
\begin{equation}
(xx)(yz)=(xy)(xz),\qquad (yz)(xx)=(yx)(zx),  \tag{$\ast $}
\end{equation}%
and a groupoid is called `abelian' (or medial, or entropic) if $%
(xy)(zu)=(xz)(yu)$ holds (see \cite{Kepka wa qgr}). To avoid confusion with
the universal algebraic definitions, we will use the word \textit{entropic}
in the latter case. Minimal clones are always idempotent, and in this case
the identities $(\mathbb{\ast })$ are equivalent to the \textit{distributive
identities}:

\begin{tabular}{ll}
$\text{\emph{Left distributivity:}}$ & $x(yz)=(xy)(xz),$ \\ 
\emph{Right}$\text{\emph{\ distributivity:}}$ & $(yz)x=(yx)(zx).$%
\end{tabular}

\noindent Any idempotent abelian groupoid is entropic (\cite{KK}, Theorem
3.2), and one might expect that idempotent weakly abelian groupoids are
distributive. We do not know if this is true or not, but for our present
purposes the weaker properties stated in the next two lemmas are sufficient.

\begin{lemma}
\label{LEMMA 3.1}If $\mathbb{A}$ is an idempotent weakly abelian groupoid,
then $uv_{1}=uv_{2}=w$ implies $u(v_{1}v_{2})=w$, i.e. $\{v\mid uv=w\}$ is a
subuniverse for any given $u,w\in \mathbb{A}$.
\end{lemma}

\begin{proof}
Applying the definition of weak abelianness with $\mathbf{a}=(u,v_{1},u),%
\mathbf{b}=(u,u,v_{1})$, $\mathbf{c}=v_{1},\mathbf{d}=v_{2}$ for $%
t(x_{1},x_{2},x_{3},x_{4})=(x_{1}x_{2})(x_{3}x_{4})$ we get 
\begin{equation*}
\begin{pmatrix}
(uv_{1})(uv_{1}) & (uv_{1})(uv_{2}) \\ 
(uu)(v_{1}v_{1}) & (uu)(v_{1}v_{2})%
\end{pmatrix}%
=%
\begin{pmatrix}
ww & ww \\ 
uv_{1} & u(v_{1}v_{2})%
\end{pmatrix}%
=%
\begin{pmatrix}
w & w \\ 
w & u(v_{1}v_{2})%
\end{pmatrix}%
\in \mathcal{M}(\mathbb{A}),
\end{equation*}%
hence $u(v_{1}v_{2})=w$.
\end{proof}

\begin{lemma}
\label{LEMMA 3.2}Any idempotent weakly abelian groupoid satisfies the
following identities:%
\renewcommand{\theenumi}{(\roman{enumi})}
\renewcommand{\labelenumi}{\theenumi}%

\begin{enumerate}
\item \label{LEMMA 3.2 i}$(xy)(xz)=(x(yz))((xy)(xz))$;

\item \label{LEMMA 3.2 ii}$(yx)(zx)=((yx)(zx))((yz)x)$;

\item \label{LEMMA 3.2 iii}$(xy)x=x(yx)$.
\end{enumerate}
\end{lemma}

\begin{proof}
Let $\mathbb{A}$ be an idempotent weakly abelian groupoid. To prove \ref%
{LEMMA 3.2 i}, we will use the 8-ary term $((\cdot \cdot )(\cdot \cdot
))((\cdot \cdot )(\cdot \cdot ))$; the underlined letters show the entries
occupied by $\mathbf{c}$ and $\mathbf{d}$ in the definition. We have%
\begin{multline*}
\begin{pmatrix}
((xy)(x\underline{y}))((x\underline{x})(\underline{z}z)) & \!\!((xy)(x%
\underline{z}))((x\underline{y})(\underline{x}z)) \\ 
((xx)(y\underline{y}))((x\underline{x})(\underline{z}z)) & \!\!((xx)(y%
\underline{z}))((x\underline{y})(\underline{x}z))%
\end{pmatrix}
\\
=%
\begin{pmatrix}
(xy)(xz) & \hfill (xy)(xz) \\ 
(xy)(xz) & \!\!(x(yz))(xy)(xz)%
\end{pmatrix}%
\in \mathcal{M}(\mathbb{A}),
\end{multline*}%
therefore the equality in \ref{LEMMA 3.2 i} holds. Doing the same with the
dual $\langle A,yx\rangle $ of $\mathbb{A}=\langle A,xy\rangle $, which is
of course also weakly abelian, we obtain the second identity. We could
derive the third identity in a similar manner, but it is easier to deduce it
from the previous ones. If we put $z=x$ in \ref{LEMMA 3.2 i} we get $%
(xy)x=(x(yx))((xy)x)$; writing $y=x$ and $z=y$ in \ref{LEMMA 3.2 ii} yields $%
x(yx)=(x(yx))((xy)x)$; comparing them gives \ref{LEMMA 3.2 iii}.
\end{proof}

In light of the last identity we will sometimes omit the parentheses in a
product of the form $xyx$. To make the connection between distributivity and
weak abelianness more explicit, we will define a relation $\sim $ on our
groupoid by $a\sim b$ iff $ab=a$. Identity \ref{LEMMA 3.2 ii} says that $%
\mathbb{A}$ is right distributive `modulo $\sim $'. This does not make
perfect sense yet, since $\sim $ may not be an equivalence relation. Our
strategy will be to reduce the problem to the case when $\sim $ is a
congruence relation. As a preparation, we first show that assuming that the
clone of $\mathbb{A}$ is minimal, we can conclude that $\mathbb{A}$
satisfies at least one-sided distributivity.

\begin{lemma}
\label{LEMMA 3.3}A weakly abelian groupoid with a minimal clone must satisfy
at least one of the distributive laws.
\end{lemma}

\begin{proof}
Suppose that $\mathbb{A}$ is a weakly abelian groupoid with a minimal clone,
and $\mathbb{A}$ is neither left nor right distributive. First we will show
that there is a two-element left zero semigroup in $\mathcal{V}(\mathbb{A})$%
. Since $\mathbb{A}$ is not right distributive, we can find elements $x,y,z$
such that $b=(yz)x\neq (yx)(zx)=a$. The second identity of Lemma \ref{LEMMA
3.2} shows that $ab=a$. If $ba=b$, then $\{a,b\}$ is a two-element left zero
subsemigroup of $\mathbb{A}$. If $ba\neq b$, then let $c$ denote the product 
$ba$, which is different from $a$ by the weak abelian property (see the
figure after Definition \ref{DEF 1.1}). We have $ab=aa=a$, so Lemma \ref%
{LEMMA 3.1} yields that $a=a(ba)=ac$. With the help of identity \ref{LEMMA
3.2 iii} of Lemma \ref{LEMMA 3.2} we can compute $cb=(ba)b=b(ab)=ba=c$. Thus
we have the following part in the multiplication table of $\mathbb{A}$.%
\begin{equation*}
\begin{tabular}{l|lll}
& $a$ & $b$ & $c$ \\ \hline
$a$ & $a$ & $a$ & $a$ \\ 
$b$ & $c$ & $b$ &  \\ 
$c$ &  & $c$ & $c$%
\end{tabular}%
\end{equation*}%
If $bc=b$, then again we have a two-element left zero subsemigroup, $\{b,c\}$%
. Suppose therefore that $bc\neq b$. Then $x(xy)$ is a nontrivial operation,
since $a(ab)=aa=a\neq b$ and $b(ba)=bc\neq b$. However, the operation $x(xy)$
is trivial on the set $\{a,c\}$. The only entry which we need to verify is $%
c(ca)=c$. We can get this equality by simply applying the definition of weak
abelianness on the following matrix:%
\begin{equation*}
\begin{pmatrix}
c(\underline{b}b) & c(\underline{c}b) \\ 
c(\underline{b}a) & c(\underline{c}a)%
\end{pmatrix}%
=%
\begin{pmatrix}
c & c \\ 
c & c(ca)%
\end{pmatrix}%
\in \mathcal{M}(\mathbb{A}).
\end{equation*}%
Therefore any operation in the clone generated by $x(xy)$ is a first
projection on $\{a,c\}$, and the original multiplication must be in this
clone since it was supposed to generate a minimal clone. Thus we have $ca=c$%
, that is, $\{a,c\}$ is a two-element left zero subsemigroup.

Passing from $\mathbb{A}$ to its dual, which is not left or right
distributive (since $\mathbb{A}$ itself is not right or left distributive)
we see from the fact proved in the preceding paragraph that $\mathbb{A}$
also has a two-element right zero subsemigroup. The product of these two is
a nontrivial rectangular band in $\mathcal{V}(\mathbb{A})$, therefore
Theorem \ref{THM 2.3} implies that $\mathbb{A}$ itself is a rectangular
band. This is a contradiction, since rectangular bands are distributive.
\end{proof}

With the help of Lemma \ref{LEMMA 3.3} we will be able to handle all cases
where $\sim $ is not a congruence relation, and finally we will arrive at
the quotient groupoid $\mathbb{A}/\!\!\sim $, which will turn out to be
distributive. This will be a rather lengthy argument, so we postpone it to
the next section. Here we give the characterization of distributive
groupoids with a minimal clone, which we will need to analyse $\mathbb{A}%
/\!\!\sim $. We will use the classification of entropic groupoids with a
minimal clone (cf.\cite{KSz comm}). To state this result, we need to define
the following varieties.

An idempotent semigroup is called a \textit{left normal band} if it
satisfies the identity $xyz=xzy$; similarly \textit{right normal bands} are
those satisfying the identity $xyz=yxz$. The variety of \textit{normal bands}
is the join of these two varieties. A groupoid is called a \textit{right
semilattice} if it satisfies the identities $xx=x, x(yz)=xy, (xy)z=(xz)y$
and $(xy)y=xy$. The dual of a right semilattice is a \textit{left semilattice%
}.

Now we can describe the entropic groupoids which have a minimal clone. (Note
that the statement is slightly different from Theorem 3.20 in \cite{KSz comm}%
, because here we formulate the description in terms of concrete clones
instead of abstract clones.)

\begin{theorem}
\label{THM 3.4}%
{\upshape{(\cite{KSz comm}).}}
Let $\mathbb{A}$ be an entropic groupoid with a minimal clone. Then $\mathbb{%
A}$ or its dual is an affine space, a rectangular band, a left normal band,
a right semilattice or a p-cyclic groupoid.
\end{theorem}

Let us turn to the investigation of distributive groupoids with a minimal
clone. It was shown in \cite{Kepka-Nemec} that every distributive groupoid
is trimedial, i.e. any subgroupoid generated by at most three elements is
entropic. The next theorem shows that the distributive and entropic
properties are equivalent for groupoids with a minimal clone.

\begin{theorem}
\label{THM 3.5}If $\mathbb{A}$ is a distributive groupoid with a minimal
clone, then the entropic law holds in $\mathbb{A}$.
\end{theorem}

\begin{proof}
We know that all three-generated subgroupoids of $\mathbb{A}$ are entropic.
If they are all trivial, then there must be a left and a right zero
semigroup among them (since the clone of $\mathbb{A}$ is not trivial), and
the product of these gives a nontrivial rectangular band in $\mathcal{V(%
\mathbb{A})}$. Applying Theorem \ref{THM 2.3}, we get that $\mathbb{A}$ is a
rectangular band. If there is a nontrivial 3-generated subalgebra which is
an affine space, a rectangular band, or (the dual of) a $p$-cyclic groupoid,
then again by Theorem \ref{THM 2.3} we have that $\mathbb{A}$ (or its dual)
belongs to one of these varieties. Hence in all these cases $\mathbb{A}$ is
entropic.

So we can assume that every three-generated subgroupoid of $\mathbb{A}$ is a
left or right semilattice or a normal band. If there is a nontrivial right
semilattice among them, then the term $x(xy)$ is the first projection on
this subalgebra, hence by the minimality of the clone we have that $\mathbb{A%
}\models $ $x(xy)=x$. This equation does not hold in a left semilattice or
in a normal band, except for a left zero semigroup (which is a right
semilattice). Thus we have that every 3-generated subalgebra is a right
semilattice. This means that all identities involving at most three
variables which hold in the variety of right semilattices also hold in $%
\mathbb{A}$. Since right semilattices are axiomatizable by three-variable
identities, we conclude that $\mathbb{A}$ itself is a right semilattice.

The case of left semilattices is similar, so finally we can suppose that we
have only normal bands as 3-generated subalgebras, i.e. that $\mathbb{A}$
satisfies all 3-variable identities that hold for normal bands.
Associativity is such an identity, so our groupoid is a distributive
semigroup, hence entropic (cf.\cite{Kepka-Nemec}, Proposition 2.3).
\end{proof}

Finally let us see which of the varieties mentioned in Theorem \ref{THM 3.4}
contain nontrivial weakly abelian algebras.

\begin{theorem}
\label{THM 3.6}If $\mathbb{A}$ is a weakly abelian entropic groupoid with a
minimal clone, then $\mathbb{A}$ or its dual is a rectangular band, an
affine space or a $p$-cyclic groupoid.
\end{theorem}

\begin{proof}
By Theorem \ref{THM 3.4}, we only need to show that $\mathbb{A}$ cannot be a
left or right normal band, or left or right semilattice. A nontrivial
semilattice is clearly not weakly abelian. In a nontrivial left normal band
one can find elements $a,b$ such that $a\neq ab$. It is easy to check that $%
\{a,ab\}$ is a two-element subsemilattice, contradicting weak abelianness.
Similarly, a nontrivial right normal band cannot be weakly abelian either.

Finally, let us suppose that $\mathbb{A}$ is a right semilattice (the case
of a left semilattice is similar). Considering the matrix%
\begin{equation*}
\begin{pmatrix}
(x\underline{y})(yy) & (x\underline{x})(yy) \\ 
(x\underline{y})(xy) & (x\underline{x})(xy)%
\end{pmatrix}%
=%
\begin{pmatrix}
xy & xy \\ 
xy & x%
\end{pmatrix}%
\in \mathcal{M(\mathbb{A})}
\end{equation*}%
we see that $xy=x$ holds for all $x,y\in A$, and this contradicts the
assumption that $\mathbb{A}$ has a minimal clone.
\end{proof}

\bigskip

\section{Left distributive weakly abelian groupoids with minimal clones}

\bigskip

Throughout this section $\mathbb{A}$ will denote a weakly abelian groupoid
with a minimal clone. Lemma \ref{LEMMA 3.3} shows that such a groupoid
satisfies at least one of the distributive laws, so we will suppose that $%
\mathbb{A}$ is left distributive. We define a binary relation $\sim $ on $%
\mathbb{A}$ by $a\sim b$ iff $ab=a$. Clearly, this relation is reflexive. In
a sequence of lemmas we will prove that if $\sim $ is not a congruence, then 
$\mathbb{A}$ is a $p$-cyclic groupoid for some prime $p$.

\begin{lemma}
\label{LEMMA 4.1}If $\sim $ is not symmetric, then $\mathbb{A}\models
x(xy)=x $.
\end{lemma}

\begin{proof}
Suppose that there are elements $a,b\in A$ such that $a\sim b$ but $%
b\not\sim a$, that is, $ab=a$ and $ba=c\neq b$. This situation is the same
as in Lemma \ref{LEMMA 3.3}, and we will proceed similarly, but this time we
go farther. Again, we have $c\neq a$ by the weak abelian property. Let $%
\mathbb{S}$ be the subgroupoid of $\mathbb{A}$ generated by $a$ and $b$.
According to Lemma \ref{LEMMA 3.1} $\{x\mid ax=a\}$ is a subuniverse of $%
\mathbb{A}$, and it contains $a$ and $b$. Therefore it contains $\mathbb{S}$%
, which implies that $a$ is a left zero element in this subgroupoid.
Moreover, $xy=a$ implies $x=a$ for $x,y\in S$. This can be seen in the
multiplication table of $\mathbb{S}$ by weak abelianness.%
\begin{equation*}
\begin{tabular}{l|lllll}
& $a$ & $\cdots $ & $x$ & $\cdots $ & $y$ \\ \hline
$a$ & $a$ & $\cdots $ & $a$ & $\cdots $ & $a$ \\ 
$\vdots $ & $\vdots $ &  & $\vdots $ &  & $\vdots $ \\ 
$x$ & $\ast $ & $\cdots $ & $x$ & $\cdots $ & $a$%
\end{tabular}%
\end{equation*}%
(Note that we have $xx=x$ by idempotence, and $\mathbb{\ast }$ indicates $xa$%
, its value is irrelevant.)

Next we show that $c$ is almost a left zero element in $\mathbb{S}$; more
precisely, $cz=c$ for all $z\in S\setminus \{a\}$. Since $z$ is in the
subgroupoid generated by $a$ and $b$, there is a binary term $t$ such that $%
t(a,b)=z$. We prove $cz=c$ by induction on the length of $t$. If this length
is zero, then either $t(x,y)=x$ or $t(x,y)=y$. The former is impossible
because $z\neq a$. In the latter case we have $cb=(ba)b=b(ab)=ba=c$. Now for
the induction step suppose that $z=t(a,b)=uv$ with $%
u=t_{1}(a,b),v=t_{2}(a,b) $. Again, $u\neq a$ follows from $z\neq a$ and
therefore $cu=c$ by the induction hypothesis. If $v$ is also different from $%
a$, then $cv=c$, so $cz=c(uv)=c$ by Lemma \ref{LEMMA 3.1}. If $v=a$, then we
have to prove $c(ua)=c$. Let us consider the matrix%
\begin{equation*}
\begin{pmatrix}
c(b\underline{b}) & c(b\underline{a}) \\ 
c(u\underline{b}) & c(u\underline{a})%
\end{pmatrix}%
=%
\begin{pmatrix}
cb & cc \\ 
c(ub) & c(ua)%
\end{pmatrix}%
=%
\begin{pmatrix}
c & c \\ 
c(ub) & c(ua)%
\end{pmatrix}%
\in \mathcal{M}(\mathbb{A}).
\end{equation*}%
We know that $cu=cb=c$, therefore $c(ub)=c$ as before. Therefore our matrix
is of the form $\bigl(      
\begin{smallmatrix}
c & c \\ 
c & c\left( ua\right)%
\end{smallmatrix}
\bigr)$, hence $cz=c(ua)=c$ by weak abelianness.

What we just proved means that in the multiplication table of the
subgroupoid $\mathbf{S}$, the row of $c$ is constant $c$ except for $ca$
which may be different. In the same way as we proved that $xy=a$ implies $%
x=a $, we can show that $xy=c$ implies $x=c$ or $y=a$, that is, $c$ can
appear only in its own row and in the column of $a$.

The knowledge we gathered about the multiplication table is enough to see
that the operation $x(xy)$ preserves $S\setminus \{c\}$. Indeed, if $x(xy)=c$
for some $x,y\in S$, then either $x=c$ or $xy=a$. The latter is impossible
since it would force $x=a$, but then $x(xy)=a\neq c$. However, the original
multiplication does not preserve this set because $ab=c$. Therefore by the
minimality of the clone, $x(xy)$ must be a projection. Since $a(ab)=a\neq b$%
, it can only be the first projection, i.e. the identity $x(xy)=x$ holds in $%
\mathbb{A}$.
\end{proof}

\begin{lemma}
\label{LEMMA 4.2}If $\sim $ is symmetric but not transitive, then $\mathbb{A}%
\models x(xy)=x$.
\end{lemma}

\begin{proof}
Suppose that there are elements $a,b,c\in A$ such that $a\sim b\sim c$ but $%
a\not\sim c$. Then $a,b,c$ must be pairwise different, because $\sim $ is
reflexive by the idempotence of $\mathbb{A}$. A part of the multiplication
table looks like this:%
\begin{equation*}
\begin{tabular}{l|lll}
& $a$ & $b$ & $c$ \\ \hline
$a$ & $a$ & $a$ &  \\ 
$b$ & $b$ & $b$ & $b$ \\ 
$c$ &  & $c$ & $c$%
\end{tabular}%
\end{equation*}%
It is easy to check that we have the same in the multiplication table of $%
x(xy)$. But for this operation we can compute the missing two entries, too,
with the help of the left distributive identity:%
\begin{align*}
a(ac)& =(ab)(ac)=a(bc)=ab=a, \\
c(ca)& =(cb)(ca)=c(ba)=cb=c.
\end{align*}%
Thus we see that $x(xy)$ is the first projection on the set $\{a,b,c\}$, but
the original operation $xy$ is not, because $a\not\sim c$ implies $ac\neq a$%
. Therefore, by the minimality of the clone of $\mathbb{A}$, $x(xy)$ must be
a trivial operation, hence $\mathbb{A}$ satisfies $x(xy)=x$.
\end{proof}

To finish the investigation of the cases where $\sim $ is not an equivalence
relation, we will show that a weakly abelian groupoid with a minimal clone
satisfying $x(xy)=x$ must be a $p$-cyclic groupoid. This will be the
consequence of the following lemma, where we do not assume weak abelianness.

\begin{lemma}
\label{LEMMA 4.3}Let $\mathbb{A}$ be a groupoid with a minimal clone such
that $\mathbb{A}$ satisfies the identity $x(yz)=xy$. Then either $\mathbb{A}$
is a $p$-cyclic groupoid, or the identity $(xy)y=xy$ holds in $\mathbb{A}$.
\end{lemma}

\begin{proof}
Suppose that $t_{1},t_{2}$ are two terms, and the leftmost variable of $%
t_{2} $ is $x$. Then it can be shown easily by induction on the length of $%
t_{2}$, that the identity $t_{1}t_{2}$=$t_{1}x$ holds in $\mathbb{A}$. This
means that any term $t$ of $\mathbb{A}$ can be reduced to a left-associated
product: $t=(\cdots ((xy_{1})y_{2})\cdots )y_{n}$. Let us now compute what
happens if we multiply a term with its leftmost variable: $tx=t{\underline{t}%
}=t$ because the leftmost variable of the underlined $t$ is also $x$. Thus
we have the same situation as in Claim 3.9 of \cite{KSz comm}, except that
the order of the variables $y_{1},\ldots ,y_{n}$ is not irrelevant. However,
when we compute binary terms, we do not have to permute them, so every
binary term is of the form $xy^{k}$, and we can proceed as in \cite{KSz comm}
to show that either $(xy)y=xy$ or $xy^{p}=x$ holds for some prime number $p$%
. In the first case we are done, so let us suppose that the latter holds.
One can check that the term $t(x,y,z)=(((xy^{p-1})z)y)z^{p-1}$ satisfies the
identities $t(x,x,z)=t(x,y,x)=t(x,y,y)=x$, i.e., it is a first
semiprojection. Therefore $t$ does not generate any nontrivial binary
operation, so it must be trivial: $t(x,y,z)=x$. Substituting $xy$ for $x$ in
this equality and multiplying both sides on the right with $z$ we get the
identity $t(xy,y,z)z=(xy)z$. Computing the left hand side we obtain the
identity $(xz)y=(xy)z$. Thus all the defining identities of the variety of $%
p $-cyclic groupoids hold in $\mathbb{A}$.
\end{proof}

\begin{remark}
One might think that in the case $\mathbb{A}\models (xy)y=xy$ we can
conclude that $\mathbb{A}$ is a right semilattice, but this is not true. The
variety defined by the identities $xx=x,x(yz)=xy,(xy)y=xy$ has a minimal
clone. Indeed, any nontrivial term can be written in the form $t=(\cdots
((xy_{1})y_{2})\cdots )y_{n}$, and identifying all the $y_{i}$s we get $%
xy^{n}=xy$. However, these identities do not imply $(xy)z=(xz)y$, so the
variety of right semilattices is a proper subvariety of the above variety.
\end{remark}

\begin{lemma}
\label{LEMMA 4.4}If $\mathbb{A}$ is a weakly abelian groupoid with a minimal
clone that satisfies the identity $x(xy)=x$, then $\mathbb{A}$ is a $p$%
-cyclic groupoid.
\end{lemma}

\begin{proof}
We show that weak abelianness and the identity $x(xy)=x$ imply the stronger
identity $x(yz)=xy$. Let $t=t(x,y,z)=x(yz)$, and compute the following
matrix:%
\begin{equation*}
\begin{pmatrix}
t(t\underline{z}) & t(t\underline{y}) \\ 
x(y\underline{z}) & x(y\underline{y})%
\end{pmatrix}%
=%
\begin{pmatrix}
t & t \\ 
t & xy%
\end{pmatrix}%
\in \mathcal{M(\mathbb{A})}.
\end{equation*}

Thus we have $x(yz)=xy$ and we can apply the preceding lemma. The only thing
we need to show is that the identity $(xy)y=xy$ cannot hold. We can proceed
the same way as we did at the end of the proof of Theorem \ref{THM 3.6} to
see that $(xy)y=xy$ would imply $xy=x$.
\end{proof}

So far we have proved that if $\sim $ is not an equivalence relation, then $%
\mathbb{A}$ is a $p$-cyclic groupoid. From now on we will assume that $\sim $
is an equivalence relation, and we will force it to be a congruence of $%
\mathbb{A}$. Using the left distributive identity we can show that $\sim $
is not very far from being a congruence.

\begin{lemma}
\label{LEMMA 4.5}For any $a,b,c\in \mathbb{A}$, if $a\sim b$ then the
following relations are true:%
\renewcommand{\theenumi}{(\roman{enumi})}
\renewcommand{\labelenumi}{\theenumi}%
\begin{enumerate}
\item \label{LEMMA 4.5 i}$ca\sim cb$,

\item \label{LEMMA 4.5 ii}$(ac)(bc)\sim ac$.
\end{enumerate}
\end{lemma}

\begin{proof}
To prove \ref{LEMMA 4.5 i} we simply apply the left distributive law: $%
(ca)(cb)=c(ab)=ca$. For \ref{LEMMA 4.5 ii} we substitute $x=c,y=a,z=b$ in
the identity $(yx)(zx)=((yx)(zx))((yz)x)$, which holds in $\mathbb{A}$ by
Lemma \ref{LEMMA 3.2}. We get $(ac)(bc)=((ac)(bc))((ab)c)=((ac)(bc))(ac)$
which is just what we had to prove.
\end{proof}

It would be nice if we had $ac\sim bc$ in \ref{LEMMA 4.5 ii}, because then $%
\sim $ would be a congruence. With the next lemma we finish the
investigation of the case where $\sim $ is not a congruence.

\begin{lemma}
\label{LEMMA 4.6}If $\sim $ is not a congruence relation, then $\mathbb{A}$
is a $p$-cyclic groupoid.
\end{lemma}

\begin{proof}
We prove first that for any $a,b,c\in \mathbb{A}$, if $a\sim b$ then the
subalgebra generated by $ac$ and $bc$ satisfies the identity $x(xy)=x$. The
second part of the previous lemma shows that $uv\sim u$ holds for $u,v\in
S=\{ac,bc\}$. Next we show that this property is inherited when we pass from 
$S$ to the subgroupoid generated by $S$. This can be done using the
following two rules:%
\begin{align*}
(uw\sim u,uv\sim u)& \Rightarrow (uv)w\sim uv, \\
(wu\sim w,wv\sim w)& \Rightarrow w(uv)\sim w.
\end{align*}%
To check the first one, we calculate $u((uv)w)=(u(uv))(uw)=u(uw)=u$, which
shows that $u\sim (uv)w$. We have assumed $u\sim uv$ therefore by
transitivity and symmetry $(uv)w\sim uv$ follows. The second one is easier: $%
w(w(uv))=w((wu)(wv))=(w(wu))(w(wv))=ww=w$. With these rules one can show by
induction on the length of terms that $uv\sim u$ for all $u,v$ in the
subgroupoid generated by $S$. Hence this subgroupoid satisfies the identity $%
x(xy)=x$.

If $\sim $ is not a congruence, then we can find elements $a,b,c$ such that $%
a\sim b$ but $ac\not\sim bc$, that is, $(ac)(bc)\neq (ac)$. If $(ac)(bc)=bc$%
, then by the second part of Lemma \ref{LEMMA 4.5} we would have $bc\sim ac$%
, which is impossible since $ac\not\sim bc$. Thus the subalgebra generated
by $\{ac,bc\}$ is not trivial. Then it has a minimal clone; it is weakly
abelian, and satisfies $x(xy)=x$, therefore by Lemma \ref{LEMMA 4.4} it is a
nontrivial $p$-cyclic groupoid in $\mathcal{V}(\mathbb{A})$. With the help
of Theorem \ref{THM 2.3} we conclude that $\mathcal{V}(\mathbb{A})$ is the
variety of $p$-cyclic groupoids.
\end{proof}

Let us summarize what we have proved so far in this section.

\begin{theorem}
\label{THM 4.7}If $\mathbb{A}$ is a weakly abelian left distributive
groupoid with a minimal clone such that the relation $\sim $ defined by $%
a\sim b\Leftrightarrow ab=a$ is not a congruence, then $\mathbb{A}$ is a $p$%
-cyclic groupoid for some prime $p$.
\end{theorem}

So finally we can suppose that $\mathbb{A}$ is a left distributive weakly
abelian groupoid with a minimal clone, and $\sim $ is a congruence of $%
\mathbb{A}$. The corresponding factor groupoid $\mathbb{A}/\!\!\sim $ is
distributive; right distributivity follows, because $\mathbb{A}$ satisfies
identity \ref{LEMMA 3.2 ii} from Lemma \ref{LEMMA 3.2}. Furthermore, $%
\mathbb{A}/\!\!\sim $ has a minimal or trivial clone. Therefore it is
entropic by Theorem \ref{THM 3.5}, and it must have at least two elements,
since $\mathbb{A}$ is not trivial. Using the list of entropic groupoids with
a minimal clone, we will prove that $\mathbb{A}$ is also entropic. The key
observation is that by the definition of $\sim $ we have 
\begin{equation*}
\mathbb{A}/\!\!\sim \ \models t_{1}=t_{2}\Leftrightarrow \mathbb{A}\models
t_{1}t_{2}=t_{1}.
\end{equation*}

\begin{lemma}
\label{LEMMA 4.8}If $\mathbb{A}/\!\!\sim $ has a two-element left or right
zero subsemigroup then $\mathbb{A}$ is entropic. It is impossible to have a
two-element semilattice among the subgroupoids of $\mathbb{A}/\!\!\sim $.
\end{lemma}

\begin{proof}
First let us suppose that $X,Y\in \mathbb{A}/\!\!\sim $ form a left zero
semigroup. Then for any $x,y\in X\cup Y$ we have $xy\sim x$. Therefore $%
x(xy)=x$ holds in $X\cup Y$, which is a nontrivial subgroupoid of $\mathbb{A}
$, since $X$ and $Y$ are two different congruence classes. By Lemma \ref%
{LEMMA 4.4} this subgroupoid must be $p$-cyclic, and by the minimality of
the clone of $\mathcal{V}(\mathbb{A})$, Theorem \ref{THM 2.3} implies that $%
\mathbb{A}$ itself must also be a $p$-cyclic groupoid.

Now suppose that $X,Y\in \mathbb{A}/\!\!\sim $ form a right zero semigroup.
Again, $X\cup Y$ is a subgroupoid of $\mathbb{A}$, and $t_{1}t_{2}=t_{1}$
holds in this subalgebra whenever the rightmost variables of $t_{1}$ and $%
t_{2}$ are the same (i.e., when $t_{1}=t_{2}$ holds in right zero
semigroups). Using this fact and the weak abelian property, we can compute $%
x(yz)$ for $x,y,z\in X\cup Y$ as follows: 
\begin{equation*}
\begin{pmatrix}
((xy)\underline{y})z & ((xy)\underline{z})z \\ 
((xx)\underline{y})z & ((xx)\underline{z})z%
\end{pmatrix}%
=%
\begin{pmatrix}
(xy)z & (xy)z \\ 
(xy)z & xz%
\end{pmatrix}%
\in \mathcal{M}(\mathbb{A}),
\end{equation*}%
therefore the identity $(xy)z=xz$ holds in $X\cup Y$. Similarly, $X\cup
Y\models x(yz)=xz$ can be shown by considering the following matrix: 
\begin{equation*}
\begin{pmatrix}
(xz)(\underline{z}z) & (xz)(\underline{y}z) \\ 
(xx)(\underline{z}z) & (xx)(\underline{y}z)%
\end{pmatrix}%
=%
\begin{pmatrix}
xz & xz \\ 
xz & x(yz)%
\end{pmatrix}%
\in \mathcal{M}(\mathbb{A}).
\end{equation*}%
Thus $X\cup Y$ is a rectangular band, and if it is nontrivial, then $\mathbb{%
A}$ is also a rectangular band by Theorem \ref{THM 2.3}, so we are done. If $%
X\cup Y$ is trivial, then $X$ and $Y$ must be singletons, because $X$ and $Y$
are left zero subsemigroups. Therefore $X\cup Y$ is a right zero
subsemigroup in $\mathbb{A}$. Forming the direct product of this with any
non-singleton congruence class we get a nontrivial rectangular band in $%
\mathcal{V}(\mathbb{A})$, so $\mathbb{A}$ is also a rectangular band by
Theorem \ref{THM 2.3}. If all the $\sim $-blocks of $\mathbb{A}$ are
singletons, then $\mathbb{A}=\mathbb{A}/\!\!\sim $ is distributive, hence
entropic by Theorem \ref{THM 3.5}.

Finally, let us suppose that $X,Y\in \mathbb{A}/\!\!\sim $ form a
semilattice. Then $X\cup Y$ satisfies every equation of the form $%
t_{1}t_{2}=t_{1}$ where $t_{1}=t_{2}$ is valid in every semilattice.
Combining this with identity \ref{LEMMA 3.2 iii} from Lemma \ref{LEMMA 3.2}
allows us to conclude that the identities%
\begin{align*}
(xy)y& =((xy)y)(xy)=(xy)(y(xy))=xy, \\
(xy)x& =((xy)x)(xy)=(xy)(x(xy))=xy
\end{align*}%
hold in $X\cup Y$. Using these identities we can compute the following
matrix: 
\begin{equation*}
\begin{pmatrix}
(xy)\underline{y} & (xy)\underline{x} \\ 
(xx)\underline{y} & (xx)\underline{x}%
\end{pmatrix}%
=%
\begin{pmatrix}
xy & xy \\ 
xy & x%
\end{pmatrix}%
\in \mathcal{M}(\mathbb{A}).
\end{equation*}%
Thus $X\cup Y$ is a left zero semigroup, contradicting the fact that $X$ and 
$Y$ are two different congruence classes.
\end{proof}

\begin{theorem}
\label{THM 4.9}If $\sim $ is a congruence relation of $\mathbb{A}$, then $%
\mathbb{A}$ is entropic.
\end{theorem}

\begin{proof}
There are at least two $\sim $-classes, since otherwise $\mathbb{A}$ would
be a left zero semigroup. So $\mathbb{A}/\!\!\sim $ has at least two
elements, and if it is trivial, then we can apply the previous lemma. If
this is not the case, then $\mathbb{A}/\!\!\sim $ must belong to one of the
varieties which have entropic minimal clones. In the case of affine spaces,
rectangular bands and $p$-cyclic groupoids Theorem \ref{THM 2.3} shows that $%
\mathbb{A}$ also belongs to one of these varieties. As we have seen in the
proof of Theorem \ref{THM 3.6}, a nontrivial left or right normal band
always contains a two-element subsemilattice, but Lemma \ref{LEMMA 4.8}
shows that this is impossible for $\mathbb{A}/\!\!\sim $. Finally, let us
assume that $\mathbb{A}/\!\!\sim $ is a nontrivial right semilattice. Then
it contains elements $a,b$ such that $a\neq ab$. Using the defining
identities of the variety of right semilattices, one can check that $a$ and $%
ab$ form a two-element left zero subsemigroup in $\mathbb{A}/\!\!\sim $, so
we can apply Lemma \ref{LEMMA 4.8} again. Similarly, a nontrivial left
semilattice must contain a two-element right zero subsemigroup, so Lemma \ref%
{LEMMA 4.8} applies in this case, too.
\end{proof}

Putting together Theorems \ref{THM 4.7} and \ref{THM 4.9} with Theorem \ref%
{THM 3.6} we get the main result of this section.

\begin{theorem}
\label{THM 4.10}A left distributive weakly abelian groupoid with a minimal
clone is either a rectangular band, an affine space or (the dual of) a $p$%
-cyclic groupoid for some prime $p$.
\end{theorem}

\bigskip

\section{Summary}

\bigskip

We have seen that only minimal clones of types \ref{unary}, \ref{binary} and %
\ref{minority} can have nontrivial weakly abelian representations, and in
case of types \ref{unary} and \ref{minority} all representations are
abelian. A weakly abelian groupoid with a minimal clone is left or right
distributive by Lemma \ref{LEMMA 3.3}, thus we can apply Theorem \ref{THM
4.10} (after dualizing if necessary) to see that such a groupoid must be a
rectangular band, an affine space or (the dual of) a $p$-cyclic groupoid.
This list does not contain any new items compared to Theorem \ref{THM 2.2}.

\begin{theorem}
\label{THM 5.1}If a minimal clone has a nontrivial weakly abelian
representation, then it also has a nontrivial abelian representation.
Therefore such a clone must be a unary clone, the clone of an affine space,
a rectangular band or (the dual of) a $p$-cyclic groupoid for some prime $p$.
\end{theorem}

Unary algebras, rectangular bands and affine spaces are abelian. A $p$%
-cyclic groupoid must be weakly abelian, as we shall see in the following
lemma.

\begin{lemma}
\label{LEMMA 5.2}Every $p$-cyclic groupoid is weakly abelian.
\end{lemma}

\begin{proof}
Suppose that $\mathbb{A}$ is a $p$-cyclic groupoid for some prime number $p$%
. (Actually, we will not need the fact that $p$ is prime.) Let $t$ be a term
of $\mathbb{A}$, with arity $n+m$, and let $\mathbf{a},\mathbf{b}\in A^{n},\ 
\mathbf{c},\mathbf{d}\in A^{m}$ be such that the matrix $\bigl(      
\begin{smallmatrix}
{t(\mathbf{a},\mathbf{c})} & {t(\mathbf{a},\mathbf{d})} \\ 
{t(\mathbf{b},\mathbf{c})} & {t(\mathbf{b},\mathbf{d})}%
\end{smallmatrix}
\bigr)$ is of the form $\bigl(      
\begin{smallmatrix}
u & u \\ 
u & v%
\end{smallmatrix}
\bigr)$. As we have seen in the proof of Lemma \ref{LEMMA 4.3}, every term
of $\mathbb{A}$ can be reduced to a left-associated product, so we may
assume that $t$ is of the form $t=(\cdots ((x_{1}x_{2})x_{3})\cdots )x_{n+m}$%
. Transposing our matrix if necessary, we can suppose that the leftmost
variable is occupied by entries belonging to $\mathbf{a}$ and $\mathbf{b}$,
say $a_{1}$ and $b_{1}$. Using the identity $(xy)z=(xz)y$ we can permute the
other variables, so that the entries in the first column of the matrix are: $%
t(\mathbf{a},\mathbf{c})=a_{1}a_{2}\cdots a_{n}c_{1}c_{2}\cdots c_{m}$, and $%
t(\mathbf{b},\mathbf{c})=b_{1}b_{2}\cdots b_{n}c_{1}c_{2}\cdots c_{m}$.
(Both products are left-associated, we have omitted the parentheses.) Our
groupoid is right cancellative, since multiplication by any element on the
right is a permutation of order $p$. Therefore the equation $t(\mathbf{a},%
\mathbf{c})=t(\mathbf{b},\mathbf{c})$ implies that $a_{1}a_{2}\cdots
a_{n}=b_{1}b_{2}\cdots b_{n}$. Multiplying both sides on the right with $%
d_{1},d_{2},\cdots ,d_{m}$, we conclude that $t(\mathbf{a},\mathbf{d})=t(%
\mathbf{b},\mathbf{d})$, that is $u=v$, so $\mathbb{A}$ is weakly abelian.
\end{proof}

\begin{theorem}
\label{THM 5.3}If a minimal clone has a nontrivial weakly abelian
representation, then all representations are weakly abelian.
\end{theorem}

As the following example shows, there exist nonabelian $p$-cyclic groupoids.
Therefore the two abelianness concepts differ already for groupoids with
minimal clones.

\begin{example}
For any prime number $p$ let us define the following binary operation on the
set $\mathbb{Z}\hspace{0cm}_{p}\times \{0,1\}$:%
\begin{equation*}
(a,b)\circ (c,d)=%
\begin{cases}
\left( a+1,b\right) & \text{if }b=0\text{ and }d=1; \\ 
\left( a,b\right) & \text{otherwise.}%
\end{cases}%
\end{equation*}%
The algebra $\mathbb{A}=(\mathbb{Z}\hspace{0cm}_{p}\times \{0,1\},\circ )$
is a $p$-cyclic groupoid, therefore it is weakly abelian and has a minimal
clone. It is not abelian, as we can see from the following matrix. 
\begin{equation*}
\begin{pmatrix}
(0,1)\circ (0,0) & (0,1)\circ (0,1) \\ 
(0,0)\circ (0,0) & (0,0)\circ (0,1)%
\end{pmatrix}%
=%
\begin{pmatrix}
(0,1) & (0,1) \\ 
(0,0) & (1,0)%
\end{pmatrix}%
\in \mathcal{M}(\mathbb{A}).
\end{equation*}
\end{example}

We conclude with a remark on rectangularity and strong abelianness. A
nontrivial affine space or $p$-cyclic groupoid cannot be rectangular, but
unary algebras and rectangular bands are all strongly abelian. Thus these
two concepts coincide for concrete minimal clones.

\begin{theorem}
\label{THM 5.4}If a minimal clone has a nontrivial rectangular
representation, then it also has a nontrivial strongly abelian
representation; moreover, all representations are strongly abelian. Such a
clone must be unary, or the clone of rectangular bands.
\end{theorem}


\bigskip

\begin{thebibliography}{99}
\bibitem{KK} K. A. Kearnes, \textit{Minimal clones with abelian
representations,} Acta Sci. Math. (Szeged) \textbf{61} (1995), no. 1-4,
59--76.

\bibitem{Keith-Emil} K. A. Kearnes, E.W. Kiss, \textit{Finite algebras of
finite complexity,} Discrete Math. \textbf{207} (1999), no. 1-3, 89--135.

\bibitem{KSz comm} K. A. Kearnes, \'{A}. Szendrei, \textit{The
classification of commutative minimal clones,} Discuss. Math. Algebra
Stochastic Methods \textbf{19} (1999), no. 1, 147--178.

\bibitem{Kepka wa qgr} T. Kepka, \textit{The structure of weakly abelian
quasigroups,} Czechoslovak Math. J. \textbf{28(103)} (1978), no 2, 181--188.

\bibitem{Kepka-Nemec} T. Kepka, P. Nemec, \textit{Notes on distributive
groupoids,} Math. Nachr. \textbf{87} (1979), 93--101.

\bibitem{LLPPP} L. L\'{e}vai, P. P. P\'{a}lfy, \textit{On binary minimal
clones,} Acta Cybernet. \textbf{12} (1996), no. 3, 279--294.

\bibitem{Plonka idred} J. P\l onka, \textit{On groups in which idempotent
reducts form a chain,} Colloq. Math. \textbf{29} (1974), 87--91.

\bibitem{Plonka k-cyc} J. P\l onka, \textit{On }$\mathit{k}$\textit{-cyclic
groupoids,} Math. Japon. \textbf{30} (1985), no. 3, 371--382.

\bibitem{R 5typ} I. G. Rosenberg, \textit{Minimal clones I. The five types,}
Lectures in Universal Algebra (Szeged, 1983), Colloq. Math. Soc. J\'{a}nos
Bolyai, \textbf{43}, North-Holland, Amsterdam, 1986, 405--427.

\bibitem{SzA clUA} \'{A}. Szendrei, \textit{Clones in Universal Algebra,} S%
\'{e}minaire de Math\'{e}matiques Sup\'{e}rieures, \textbf{99}, Presses de
L'Universit\'{e} de Montr\'{e}al, 1986.

\bibitem{Taylor} W. Taylor, \textit{Characterizing Mal'cev conditions,}
Algebra Universalis \textbf{3} (1973), 351--397.
\end{thebibliography}
\end{document}